\documentclass[12pt]{article}

 \usepackage{amsmath,amssymb,amscd,amsthm,esint}
 
\usepackage{graphics,amsmath,amssymb,amsthm,mathrsfs}

\usepackage{graphics,amsmath,amssymb,amsthm,mathrsfs,amsfonts}

\usepackage{sidecap}
\usepackage{float}
\usepackage{extarrows}
\usepackage{booktabs}
\usepackage{verbatim}
\usepackage{hyperref}
\usepackage[usenames,dvipsnames]{xcolor}

\newtheorem{thm}{Theorem}[section]
\newtheorem{lemma}[thm]{Lemma}

\theoremstyle{definition}
\newtheorem{remark}[thm]{Remark}

\def\XXint#1#2#3{{\setbox0=\hbox{$#1{#2#3}{\int}$}
         \vcenter{\hbox{$#2#3$}}\kern-.5\wd0}}

\def\e{\varepsilon}

\numberwithin{equation}{section}

\begin{document}

\title{Uniform  Boundary Controllability \\ and Homogenization of
Wave Equations (Revised)}

\author{Fanghua Lin \thanks{Supported in part by NSF grant DMS-1501000.}
\qquad
Zhongwei Shen\thanks{Supported in part by NSF grant DMS-1600520.}}
\date{}
\maketitle


\begin{abstract}

We obtain sharp convergence rates, using Dirichlet correctors,
 for solutions of wave equations in a bounded domain with 
rapidly oscillating periodic coefficients.
The results are used to prove the exact boundary
controllability that is uniform in $\e$ - the scale of the microstructure,
for the projection of solutions to the 
subspace generated by  the eigenfunctions with eigenvalues less that $C\e^{-1/2}$.

\medskip


\medskip


\end{abstract}

\section{\bf Introduction}\label{section-1}

In this paper we study the exact boundary controllability, which is uniform in $\e>0$, 
of the wave operator, 
\begin{equation}\label{t-op}
\partial_t^2 +\mathcal{ L}_\e
\end{equation}
in a bounded domain,
where the elliptic operator $\mathcal{L}_\e$  is given by
\begin{equation}\label{operator}
\mathcal{L}_\e =-\text{\rm div} \big( A(x/\e)\nabla\big),
\end{equation}
and $\e>0$ is a small parameter.
Throughout we will assume that the $d\times d$ coefficient matrix $A=A(y)=(a_{ij} (y)) $
is real, bounded measurable, satisfies the ellipticity condition,
\begin{equation}\label{ellipticity}
\mu |\xi|^2 \le \langle A\xi, \xi \rangle \le \frac{1}{\mu} |\xi|^2 \quad \text{ for any } \xi \in \mathbb{R}^d,
\end{equation}
 where $\mu>0$, the symmetry condition,
\begin{equation}\label{symmetry}
a_{ij}  (y) =a_{ji} (y) \quad \text{ for } 1\le i, j\le d,
\end{equation}
and the periodicity condition,
\begin{equation}\label{periodicity}
A(y+z) =A(y) \quad \text{ for any } y \in \mathbb{R}^d \text{ and  }  z\in \mathbb{Z}^d.
\end{equation}

Let $\Omega$ be a bounded domain in $\mathbb{R}^d$.
Given initial data $(\theta_{\e, 0}, \theta_{\e, 1}) 
\in L^2(\Omega) \times H^{-1}(\Omega)$,
one is interested in finding $T>0$ and a control
$g_\e\in L^2(S_T)$ such that
the weak solution of the  evolution problem,
\begin{equation}\label{E-0}
\left\{
\aligned
 &(\partial_t^2 +\mathcal{L}_\e )v_\e   =0  \quad  \text{ in } \Omega_T=\Omega \times (0, T],\\
 &  v_\e =g_\e  \quad \text{ on } S_T=\partial\Omega \times [0, T] ,\\
  & v_\e (x, 0)  =\theta_{\e, 0} (x), \quad \partial_t v_\e (x, 0)=\theta_{\e, 1} (x) \quad \text{ for } x\in \Omega,
  \endaligned
  \right.
\end{equation}
satisfies the conditions 
\begin{equation}\label{control-1}
v_\e (x, T)= \partial_t v_\e (x, T)=0 \quad \text{ for } x\in \Omega.
\end{equation}
This  classical control problem in highly heterogeneous media was proposed by J.-L. Lions in \cite{Lions-1988}.
Let $u_\e$ be the solution of the initial-Dirichlet problem, 
\begin{equation}\label{BVP-0}
\left\{
\aligned
 &(\partial_t^2 +\mathcal{L}_\e )u_\e   =0  \quad  \text{ in } \Omega_T,\\
 &  u_\e =0  \quad \text{ on } S_T,\\
  & u_\e (x, 0)  =\varphi_{\e, 0} (x), \quad \partial_t u_\e (x, 0)=\varphi_{\e, 1} (x) \quad \text{ for } x\in \Omega,
  \endaligned
  \right.
\end{equation}
where $\varphi_{\e, 0} \in H^1_0(\Omega)$ and
$\varphi_{\e, 1} \in L^2(\Omega)$.
By the Hilbert Uniqueness Method (HUM),
 the existence of  a control $g_\e$, which is uniformly bounded in $L^2(S_T)$ for $\e>0$,
 is equivalent to the following two estimates, usually called observability inequalities,
\begin{equation}\label{upper-b}
\frac{1}{T} \int_0^T\!\!\!\int_{\partial\Omega}
|\nabla u_\e|^2\, d\sigma dt
\le C   \Big\{ \| \nabla \varphi_{\e, 0} \|_{L^2(\Omega)}^2
+ \|\varphi_{\e, 1} \|_{L^2(\Omega)}^2 \Big\},
\end{equation}
and
\begin{equation}\label{lower-b}
c   \Big\{ \|\nabla \varphi_{\e, 0} \|_{L^2(\Omega)}^2
+ \|\varphi_{\e, 1} \|_{L^2(\Omega)}^2 \Big\}
\le
\frac{1}{T}  \int_0^T\!\!\!\int_{\partial\Omega}
|\nabla u_\e|^2\, d\sigma dt,
\end{equation}
with positive constants $C$ and $c$ independent of $\e>0$
(see \cite{Lions-1988}).
However, it has been known since early 1990s that both (\ref{upper-b}) and (\ref{lower-b})
fail to hold uniformly in $\e>0$, even in the case $d=1$ \cite{Ave-1992}.
We remark that for $\e=1$ (without the periodicity condition),
a fairly complete solution of the exact boundary controllability problem 
for second-order hyperbolic equations may be found
in \cite{Bardos-1992} by C. Bardos, G. Lebeau, and J. Rauch,
using microlocal analysis.
Also see related work in \cite{Burq-1997, Bardos-2006} and references therein.

In this paper we shall show that estimates (\ref{upper-b}) and (\ref{lower-b}) hold uniformly  
if the initial data  $(\varphi_{\e, 0}, \varphi_{\e, 1})$ in (\ref{BVP-0})  are taken from a low-frequency subspace of $H^1_0(\Omega)\times L^2(\Omega)$.
More precisely, 
let $\{\lambda_{\e, k}: k=1, 2, \dots \}$ denote the sequence of Dirichlet eigenvalues  in an increasing order for
$\mathcal{L}_\e$ in $\Omega$.
Let $\{ \psi_{\e, k}: k=1, 2, \dots \}$  be a set of Dirichlet eigenfunctions in $H^1_0(\Omega)$ for $\mathcal{L}_\e$
in $\Omega$ such that $\{ \psi_{\e, k}\}$ forms an orthonormal basis for $L^2(\Omega)$ and
$\mathcal{L}_\e (\psi_{\e, k}) =\lambda_{\e, k} \psi_{\e, k}$ in $\Omega$.
Define
\begin{equation}\label{A-N}
\mathcal{A}_N 
=\Big\{ h= \sum_{\lambda_{\e, k} \le N} a_k \psi_{\e, k}:  \ a_k \in \mathbb{R}  \Big\}.
\end{equation}

\begin{thm}\label{main-theorem-2}
Assume  $A=A(y)$ satisfies conditions \eqref{ellipticity},
\eqref{symmetry} and \eqref{periodicity}. Also assume that there exists $M>0$ such that 
\begin{equation}\label{Lip-C}
| A(x)-A(y)|\le M |x-y| \quad \text{ for any } x, y\in \mathbb{R}^d.
\end{equation}
Let $\Omega$ be a bounded $C^3$ domain in $\mathbb{R}^d$.
Let $u_\e$ be a solution of \eqref{BVP-0}  with initial data $(\varphi_{\e, 0}, \varphi_{\e, 1})
\in \mathcal{A}_N \times \mathcal{A}_N$.
Then,  if $N \le C_0T^{-1} \e^{-1/2}$ for some $C_0>0$, 
the inequality  \eqref{upper-b} holds with constant $C$ depending only on $d$, $\mu$, $C_0$, $M$ and $\Omega$.
Moreover, there exist $c_0>0$ and $T_0>0$, depending only on $d$, $\mu$, $M$ and $\Omega$,
such that if $N\le c_0  T^{-1} \e^{-1/2}$ and $T\ge T_0$,
then (\eqref{lower-b}) holds with constant $C$ depending only on $d$, $\mu$, $M$ and $\Omega$.
\end{thm}

Following \cite{Castro-1999},
one may use Theorem \ref{main-theorem-2} to prove the following result on the uniform 
boundary controllability. Let $N\le \delta T^{-1} \e^{-1/2}$ and $T\ge T_0$, where $\delta=\delta(d, A, \Omega)>0$ is sufficiently small.
Given $(\theta_{\e, 0}, \theta_{\e, 1}) \in L^2(\Omega) \times H^{-1} (\Omega)$,
there exists $g_\e \in L^2(S_T)$ such that the solution of (\ref{E-0})
satisfies the conditions,
\begin{equation}\label{P-C}
P_N v_\e (x, T) =0 \quad \text{ and } \quad P_N \partial_t v_\e (x, T)=0 \quad \text{ for } x\in \Omega,
\end{equation}
where $P_N$ denotes the projection operator from $L^2(\Omega)$ or $H^{-1}(\Omega)$ 
to the space $\mathcal{A}_N$.
Moreover, the control $g_\e$ satisfies the uniform estimates,
\begin{equation}\label{upper-b-C}
c\| g_\e \|_{L^2(S_T)}
\le  \Big\{ \|  P_N \theta_{\e, 0}\|_{L^2(\Omega)}
+\| P_N \theta_{\e, 1} \|_{H^{-1}(\Omega)} \Big\}
\le C \| g_\e\|_{L^2(S_T)},
\end{equation} 
where $C>0$ and $c>0$ are independent of $\e$.
See Section \ref{section-4}.

In the case $d=1$, it was proved by C. Castro  in \cite{Castro-2000-l} that 
the estimates (\ref{upper-b}) and (\ref{lower-b}) hold uniformly 
if the initial data are taken from $\mathcal{A}_N \times \mathcal{A}_N$
and $N\le \delta \e^{-2}$, where $\delta>0$ is sufficiently small.
Also see \cite{CZ-2000-h} for the case where the initial data are taken from a subspace generated by 
the eigenfunctions with  eigenvalues greater than $C \e^{-2-\sigma}$ for some $\sigma>0$.
The approaches used in \cite{Castro-2000-l, CZ-2000-h} do not extend to the multi-dimensional case.
To the best of the authors' knowledge, the only results in the case $d\ge 2$ are found in \cite{AL-1989, Lebeau-2000}.
In \cite{AL-1989}
M. Avellaneda and the first author used the asymptotic expansion of the Poisson's kernel for the elliptic operator $\mathcal{L}_\e$ in $\Omega$
to identify the weak  limits of the controls. 
In \cite{Lebeau-2000} G. Lebeau considered 
the wave operator with oscillating density, 
$\rho (x, x/\e) \partial_t^2 -\Delta_g$, where  $\Delta_g$ is the 
Laplace operator for some fixed smooth metric, and the function
$\rho(x, y)$ is periodic in $y$.
Theorem \ref{main-theorem-2} seems to be the first result
on  the observability  inequalities (\ref{upper-b}) and (\ref{lower-b}) for wave operators with  oscillating coefficients 
$A(x/\e)$ in higher dimensions.

Let 
$$
u_\e (x, t)=\cos (\sqrt{\lambda_{\e, k}} t ) \psi_{\e, k}.
$$
Then $(\partial_t^2+\mathcal{L}_\e) u_\e=0$ in $\Omega_T$ and
$u_\e=0$ on $S_T$.
Also, $u_\e (x, 0)=\psi_{\e, k} (x)$ and $\partial_t u_\e (x, 0)=0$ for $x\in \Omega$.
Thus the inequalities (\ref{upper-b}) and (\ref{lower-b}) would  imply that
\begin{equation}\label{trace-bound}
c \lambda_{\e, k} \le \int_{\partial\Omega} |\nabla \psi_{\e, k}|^2\, d\sigma
\le C \lambda_{\e, k}.
\end{equation}
It was proved in \cite{Ave-1992, Castro-2000-l} that (\ref{trace-bound}) cannot hold uniformly in $\e>0$ and $k\ge 1$.
Counter-examples were constructed using  eigenfunctions with eigenvalues 
$
\lambda_{\e, k} \sim \e^{-2}
$
\ - the wave length of the solutions is of the order of the size of the microstructure.
Also see related work in \cite{HT-2002} by A. Hassell and T. Tao for 
Dirichlet eigenfunctions on a compact Riemannian manifold with boundary.
In \cite{KLS-eigenvalue} C. Kenig and  the present authors  proved that for $d\ge 2$, 
\begin{equation}\label{trace-bound-1}
\int_{\partial\Omega} |\nabla \psi_{\e, k}|^2\, d\sigma \le C \lambda_{\e, k} (1+ \e \lambda_{\e, k}),
\end{equation}
if $\e^2\lambda_{\e, k} \le 1$,
where $C$ is independent of $\e$ and $k$.
This, in particular, implies that the upper bound in (\ref{trace-bound})
holds if $\e \lambda_{\e, k}\le 1$.
Furthermore, it is proved in \cite{KLS-eigenvalue} that
if $\e \lambda_{\e, k} \le \delta$, where $\delta>0$ depends only on $A$ and $\Omega$,
then the lower bound in (\ref{trace-bound}) also holds uniformly in $\e$ and $k$.
These results suggest that one may be able to extend Theorem \ref{main-theorem-2}
to the case $N\le C \e^{-1}$.
But this remains unknown.
In view of the one-dimensional results in \cite{Castro-2000-l, Castro-1999},
one may conjecture further that the main conclusion in Theorem \ref{main-theorem-2}
is valid when $N\le \delta \e^{-2}$ and $\delta$ is sufficiently small.

We now describe our approach to Theorem \ref{main-theorem-2},
which is based on homogenization.
Under the assumptions (\ref{ellipticity}), (\ref{symmetry}) and (\ref{periodicity}) as
well as suitable conditions on $F$, $\varphi_{\e, 0}$ and $\varphi_{\e, 1}$,
the solution $u_\e$ of the initial-Dirichlet problem, 
\begin{equation}\label{BVP-F}
\left\{
\aligned
 &(\partial_t^2 +\mathcal{L}_\e )u_\e   =F  \quad  \text{ in } \Omega_T,\\
 &  u_\e =0  \quad \text{ on } S_T,\\
  & u_\e (x, 0)  =\varphi_{\e, 0} (x), \quad \partial_t u_\e (x, 0)=\varphi_{\e, 1} (x) \quad \text{ for } x\in \Omega,
  \endaligned
  \right.
\end{equation}
converges strongly in $L^2 (\Omega_T)$ to 
the solution of the homogenized problem, 
\begin{equation}\label{BVP-0F}
\left\{
\aligned
 &(\partial_t^2 +\mathcal{L}_0 )u_0   =F  \quad  \text{ in } \Omega_T,\\
 &  u_0 =0  \quad \text{ on } S_T,\\
  & u_0 (x, 0)  =\varphi_0 (x), \quad \partial_t u_0 (x, 0)=\varphi_1 (x) \quad \text{ for } x\in \Omega,
  \endaligned
  \right.
\end{equation}
where $\mathcal{L}_0$ is an elliptic operator with constant coefficients
(see e.g. \cite{BLP}).
In the first part of this paper we shall investigate the problem of
convergence rates.

Let $\Phi_\e =(\Phi_{\e, 1}, \Phi_{\e, 2}, \dots, \Phi_{\e, d} )$ denote the Dirichlet corrector for 
the operator $\mathcal{L}_\e$ in $\Omega$, where, for $1\le j \le d$,
the function $\Phi_{\e, j}$ is the solution in $H^1(\Omega)$  of the Dirichlet problem,
\begin{equation}\label{DC}
\left\{
\aligned 
\mathcal{L}_ \e (\Phi_{\e, j}) & =0 & \quad & \text{ in } \Omega,\\
\Phi_{\e, j}  & = x_j & \quad & \text{ on } \partial\Omega.
\endaligned
\right.
\end{equation}

\begin{thm}\label{main-theorem-1}
Assume $A=A(y)$ satisfies conditions \eqref{ellipticity}, \eqref{symmetry} and \eqref{periodicity}.
Let $u_\e$ be a weak solution of \eqref{BVP-F},
where $\Omega$ is a bounded Lipschitz domain in $\mathbb{R}^d$.
Let 
\begin{equation}\label{w-0}
w_\e =u_\e -u_0 - \big( \Phi_{\e, k} -x_k \big) \frac{\partial u_0}{\partial x_k},
\end{equation}
where  $u_0$ is  the solution of \eqref{BVP-0F}. 
Then for any $t\in (0, T]$,
\begin{equation}\label{main-estimate-1}
\aligned
 & \left(\int_\Omega \big( |\nabla w_\e (x, t)|^2 +|\partial_t w_\e(x, t) |^2 \big)\, dx \right)^{1/2}\\
& \qquad
\le C \Big\{ \|\mathcal{L}_\e (\varphi_{\e, 0}) -\mathcal{L}_0 (\varphi_0)\|_{H^{-1}(\Omega)}
+ \| \varphi_{ \e, 1 } -\varphi_1 \|_{L^2(\Omega)} \Big\} \\
&\qquad\qquad
 + C \e
\Big\{ \|\nabla^2 \varphi_0 \|_{L^2(\Omega)}
+ \|\nabla \varphi_1 \|_{L^2(\Omega)} \Big\}\\
&\qquad\qquad
+ C \e \sup_{t\in (0, T]}  \|  \nabla^2  u_0 (\cdot, t) \|_{L^2(\Omega)}\\
&\qquad\qquad
 + C \e  {T}  \sup_{t\in (0, T]}  \| |\partial_t \nabla^2 u_0 (\cdot, t) | +| \partial_t^2 \nabla u_0 (\cdot, t)| \|_{L^2(\Omega)},
\endaligned
 \end{equation}
 where $C$ depends only on $d$ and $\mu$.
\end{thm}

Theorem \ref{main-theorem-1},
together with Rellich identities,
allows us to control the boundary integral 
$$
\int_0^T\!\!\!\int_{\partial\Omega}
|\nabla u_\e - (\nabla \Phi_\e) (\nabla u_0)|^2\, d\sigma dt,
$$
where the initial data $(\varphi_0, \varphi_1)$ in (\ref{BVP-0F})
is chosen so that $\mathcal{L}_0 (\varphi_0)=\mathcal{L}_\e (\varphi_{\e, 0})$ and 
$\varphi_1 =\varphi_{\e, 1}$ in $\Omega$
(see \cite{PP-2006} for the case $d=1$).
Since $|\nabla \Phi_\e|\le C$ \cite{AL-1987} and $|\text{det}(\nabla \Phi_\e)|\ge c>0$ on $\partial\Omega$ \cite{KLS-eigenvalue},
this reduces the problem to the estimates (\ref{upper-b}) and (\ref{lower-b})
for the homogenized operator $\partial_t^2 +\mathcal{L}_0$ with constant coefficients.
We remark that the Rellich identities, which use  the Lipschitz condition \eqref{Lip-C},
are applied to the function $w_\e$ in \eqref{w-0}.
We  further point out  that the power of $\e$ in the condition $N\le C_0T^{-1} \e^{-1/2}$
is dictated by the highest-order term in the right-hand side of (\ref{main-estimate-1}).
Also, the $C^3$ condition on $\Omega$ is only used for estimates of the homogenized solutions.

The problem of convergence rates is
of much interest in its own right in the theory of  homogenization. 
Note that no smoothness condition on $A$ is needed in Theorem \ref{main-theorem-1}.
Let $w_\e$ be given by (\ref{w-0}).
Since $\|\Phi_{\e} -x \|_{L^\infty(\Omega)}\le C \e$,
it follows that
\begin{equation}\label{u-t}
|\partial_t u_\e -\partial_t u_0|
\le |\partial_t w_\e|
+C \e |\partial_t \nabla u_0|,
\end{equation}
and
\begin{equation}\label{n-u}
| \nabla u_\e   - (\nabla \Phi_{\e}) (\nabla u_0)| 
\le |\nabla w_\e| + C \e |\nabla^2 u_0|,
\end{equation}
where $C$ depends only on $d$ and $\mu$.
As a result,
Theorem \ref{main-theorem-1} gives the $O(\e)$ convergence rates for both
 $\|\partial_t u_\e - \partial_t u_0\|_{L^2(\Omega)}$ and
$\|\nabla u_\e  - (\nabla \Phi_\e) (\nabla u_0) \|_{L^2(\Omega)} $.
By Sobolev imbedding, we may also deduce an $O(\e)$ convergence rate for
$\|u_\e (\cdot, t) -u_0 (\cdot,t)\|_{L^2(\Omega)}$ directly 
from (\ref{main-estimate-1}).
However, a better estimate with lower order derivatives required for $u_0$
is obtained at the end of Section \ref{section-3} (see (\ref{3-41})).
We mention that
in the case $\Omega=\mathbb{R}^d$,
the following estimate was proved in \cite{Suslina-2018-W} by M.A. Dorodnyi and T.A. Suslina,
\begin{equation}
\| u_\e (\cdot, t) -u_0 (\cdot, t) \|_{L^2(\mathbb{R}^d)}
\le C \e (t+1)
\Big\{ \|\varphi_0 \|_{H^{3/2}(\mathbb{R}^d)}
+ \| \varphi_1 \|_{H^{1/2} (\mathbb{R}^d)} \Big\}
\end{equation}
for any $t\in \mathbb{R}$,
where $(\partial_t^2 +\mathcal{L}_\e) u_\e
=(\partial_t^2 +\mathcal{L}_0) u_0=0$ 
in $\mathbb{R}^{d+1}$,
and $u_\e$ and $u_0$ have the same initial data $(\varphi_0, \varphi_1)$.
The results in \cite{Suslina-2018-W} (also see \cite{Suslina-2017}) 
 are obtained by an operator-theoretic approach,
using the Floquet-Bloch theory.
In the case of bounded domains, for a periodic hyperbolic system,
 Yu. M. Meshkova obtained an $O(\e)$ estimate for
$ \| u_\e (\cdot, t) -u_0 (\cdot, t) \|_{L^2(\Omega)}$, assuming the initial data $(\varphi_0, \varphi_1)$
belong to some subspace of $H^4(\Omega)$ \cite{Meshkova-2018}.
We note that the highest-order term in the right-hand side of (\ref{3-41})
involves $\|\varphi_0\|_{H^2(\Omega)} $
and $\| \varphi_1 \|_{H^1(\Omega)}$.

We point out  that the symmetry condition \eqref{symmetry} is essential in the proofs of Theorems \ref{main-theorem-2} and
\ref{main-theorem-1}, but the assumption that equations are scalar is not.
Theorem \ref{main-theorem-2} continues to hold for elliptic systems $\partial_t - \text{\rm div} (A(x/\e)\nabla )$,
 if $A(y) =(a_{ij}^{\alpha\beta} (y) )$, 
with $1\le i, j \le d$ and $1\le \alpha, \beta \le m$, 
satisfies the ellipticity condition \eqref{ellipticity} for $\xi =(\xi_i^\alpha)\in \mathbb{R}^{m\times d}$, the periodicity condition  \eqref{periodicity}, 
the Lipschitz condition \eqref{Lip-C},  
and the symmetry condition 
$a_{ij}^{\alpha\beta} =a_{ji}^{\beta \alpha}$.
In the case of Theorem \ref{main-theorem-1}, the estimate \eqref{main-estimate-1} holds in a $C^{1, \eta}$ domain $\Omega$,  if
$A$ satisfies \eqref{ellipticity}, \eqref{periodicity}, the symmetry condition above, and is H\"older continuous.
The additional smoothness conditions on $A$ and $\Omega$ are used for the estimates of correctors $\chi$ and $\Phi_\e$.

The summation convention that repeated indices are summed is used throughout the paper.
Finally, we thank Mathias Sch\"affner, who pointed out a flaw in the previous version of this paper. 



\section{Preliminaries}\label{section-2}

Throughout this section we will assume that $A=A(y)$ satisfies conditions 
(\ref{ellipticity}), (\ref{symmetry}) and (\ref{periodicity}).
A function $u$ in $ \mathbb{R}^d$  is said to be 1-periodic if $u(y+z)=u(y)$ for a.e.
$y\in \mathbb{R}^d$ and for any  $z\in \mathbb{Z}^d$.
Let  $\mathbb{T}^d =\mathbb{R}^d/\mathbb{Z}^d \cong [0, 1)^d$.
We use $H^1 (\mathbb{T}^d)$ to denote the closure of 1-periodic $C^\infty$ functions in $\mathbb{R}^d$ in the space
$H^1(Y)$, where $Y=(0, 1)^d$.

Let $\chi (y) =(\chi_1 (y), \chi_2 (y), \dots, \chi_d (y))$ 
denote the first-order corrector for $\mathcal{L}_\e$,
where, for $1\le j\le d$,  the function $\chi_j=\chi_j (y)$ is the unique 
weak solution in $H^1(\mathbb{T}^d) $ of the cell problem,
\begin{equation}\label{cell}
\left\{
\aligned
& -\text{\rm div} \big( A(y)\nabla \chi_j \big) =\text{\rm div} \big(A(y) \nabla y_j\big) \quad \text{ in } \mathbb{T}^d,\\
&\int_{\mathbb{T}^d} \chi_j \, dy=0.
\endaligned
\right.
\end{equation}
Note that $\chi_j$ is 1-periodic and
\begin{equation}
\mathcal{L}_\e \big\{ x_j +\e \chi_j (x/\e) \big\}  =0  \quad \text{ in } \mathbb{R}^d.
\end{equation}
By the classical De Giorgi - Nash estimate,
$\chi_j \in L^\infty (\mathbb{R}^d)$ and  $\|\chi_j\|_\infty\le C$,
where $C$ depends only on $d$ and $\mu$. 
Let 
\begin{equation}\label{L-0}
\mathcal{L}_0 =-\text{\rm div} \big( \widehat{A}\nabla \big),
\end{equation}
where $\widehat{A} =\big(\widehat{a}_{ij} \big)_{d\times d}$ and 
\begin{equation}\label{homo-c}
\widehat{a}_{ij} 
=\int_{\mathbb{T}^d}
\left( a_{ij} + a_{ik} \frac{\partial \chi_j}{\partial y_k } \right) dy
\end{equation}
(the summation convention is used).
Under the conditions (\ref{ellipticity}), (\ref{symmetry}) and (\ref{periodicity}),
one may show that the matrix $\widehat{A}$ is symmetric and satisfies the ellipticity condition,
\begin{equation}\label{ellipticity-0}
\mu |\xi|^2 \le \langle \widehat{A} \xi, \xi \rangle \le \frac{1}{\mu} |\xi|^2 \quad \text{ for any } \xi \in \mathbb{R}^d,
\end{equation}
with the same constant $\mu$ as in (\ref{ellipticity}). 
It is well known that the homogenized operator for $\partial_t^2 +\mathcal{L}_\e$ is
given by $\partial_t^2 +\mathcal{L}_0$.
In particular,
if $\varphi_{\e, 0} =\varphi_0$ and $\varphi_{\e, 1} =\varphi_1$,
 the solution $u_\e$ of the initial-Dirichlet problem (\ref{BVP-F})
converges strongly in $L^2(\Omega_T)$ to the solution $u_0$ of the homogenized  problem (\ref{BVP-0F}).

For $1\le i, j\le d$, let
\begin{equation}\label{b}
b_{ij} =a_{ij} + a_{ik} \frac{\partial \chi_j}{\partial y_k} -\widehat{a}_{ij}.
\end{equation}
It follows by the definitions of $\chi_j$ and $\widehat{a}_{ij}$ that
\begin{equation}\label{b-p}
\frac{\partial}{\partial y_i} b_{ij} =0 \quad \text{ and } \quad
\int_{\mathbb{T}^d} b_{ij}\, dy=0.
\end{equation}

\begin{lemma}\label{lemma-2.1}
There exist 1-periodic functions $\phi_{kij} (y)$ in $H^1(\mathbb{T}^d)$  for
$1\le i, j, k\le d$ such that
$\int_{\mathbb{T}^d} \phi_{kij} \, dy=0$, 
\begin{equation}\label{b-phi}
b_{ij} =\frac{\partial}{\partial y_k} \phi_{kij} 
\quad \text{ and } \quad
\phi_{kij} = -\phi_{ikj}.
\end{equation}
Moreover, $\phi_{kij} \in L^\infty(\mathbb{R}^d)$ and
$\| \phi_{kij} \|_\infty \le C$, where $C$ depends only on $d$ and $\mu$.
\end{lemma}

\begin{proof}
See \cite[Remark2.1]{KLS-2014}.
\end{proof}

Let $\Phi_\e (x)$ be the Dirichlet corrector for $\mathcal{L}_\e$ in $\Omega$, defined by (\ref{DC}).
Since
\begin{equation}\label{L-phi}
\mathcal{L}_\e \big\{ \Phi_{\e, j} -x_j -\e \chi_j (x/\e) \big\} =0 
\quad \text{ in } \Omega,
\end{equation}
by the maximum principle,
$$
\| \Phi_{\e, j} -x_j -\e \chi_j (x/\e)\|_{L^\infty (\Omega)}
= \| \e \chi_j (x/\e)\|_{L^\infty(\partial\Omega)}.
$$
It follows that 
\begin{equation}
\|\Phi_{\e, j} -x_j \|_{L^\infty(\Omega)}
\le 2\e \|\chi_j \|_\infty\le C \e,
\end{equation}
where $C$ depends only on $d$ and $\mu$.
If $\Omega$ is a bounded $C^{1, \alpha}$ domain in $\mathbb{R}^d$ for some $\alpha>0$
and $A$ is H\"older continous,
by  the boundary Lipschitz estimate  for $\mathcal{L}_\e$ \cite{AL-1987}, we also have
\begin{equation}\label{Lip}
\|\nabla \Phi_{\e, j} \|_{L^\infty(\Omega)} 
\le C,
\end{equation}
where  $C$ depends only on $d$, $A$ and $\Omega$.

\begin{lemma}\label{lemma-2.2}
Suppose that
\begin{equation}\label{2.1-0}
(\partial_t^2 +\mathcal{L}_\e) u_\e
= (\partial_t^2 +\mathcal{L}_0) u_0
\quad \text{ in } \Omega \times (T_0, T_1).
\end{equation}
Let
\begin{equation}\label{w}
w_\e
=u_\e -u_0 - \left(\Phi_{\e, k} -x_k\right) \frac{\partial u_0}{\partial x_k} .
\end{equation}
Then
\begin{equation}\label{L-w}
\aligned
(\partial_t^2 +\mathcal{L}_\e) w_\e
= & -\e \frac{\partial}{\partial x_i}
\left\{ \phi_{kij} (x/\e) \frac{\partial^2 u_0}{\partial x_k \partial x_j} \right\}\\
& +\frac{\partial}{\partial x_i}
\left\{ a_{ij}(x/\e) \big[ \Phi_{\e, k} -x_k\big]  \frac{\partial^2 u_0}{\partial x_j \partial x_k} \right\}\\
& +a_{ij}(x/\e) \frac{\partial}{\partial x_j}
\Big[ \Phi_{\e, k} -x_k -\e \chi_k (x/\e) \Big] \frac{\partial^2 u_0}{\partial x_i \partial x_k} \\
& - \big( \Phi_{\e,k} -x_k \big) \partial_t^2 \frac{\partial u_0}{\partial x_k}.
\endaligned
\end{equation}
\end{lemma}

\begin{proof}
Note that by (\ref{2.1-0}),
$$
\aligned
(\partial_t^2 +\mathcal{L}_\e) w_\e
 &=  (\mathcal{L}_0 -\mathcal{L}_\e ) u_0 
-\mathcal{L}_\e \Big\{ (\Phi_{\e, k} -x_k) \frac{\partial u_0}{\partial x_k} \Big\}
- \big( \Phi_{\e,k} -x_k \big) \partial_t^2 \frac{\partial u_0}{\partial x_k}\\
& =\frac{\partial}{\partial x_i} 
\left\{ b_{ij} (x/\e) \frac{\partial u_o}{\partial x_j} \right\}
+\frac{\partial}{\partial x_i}
\left\{ a_{ij} (x/\e) \frac{\partial}{\partial x_j}
\Big[ \Phi_{\e, k} -x_k -\e \chi_k (x/\e) \Big] \frac{\partial u_0}{\partial x_k} \right\}\\
 & \qquad + \frac{\partial }{\partial x_i}
\left\{ a_{ij} (x/\e) \big[ \Phi_{\e, k} -x_k \big] \frac{\partial^2 u_0}{\partial x_j \partial x_k} \right\}
 - \big( \Phi_{\e,k} -x_k \big) \partial_t^2 \frac{\partial u_0}{\partial x_k},
\endaligned
$$
where $b_{ij}(y) $ is given by (\ref{b}).
Since $\frac{\partial}{\partial y_i} b_{ij}=0$, we see that 
$$
\aligned
\frac{\partial}{\partial x_i} 
\left\{ b_{ij} (x/\e) \frac{\partial u_o}{\partial x_j} \right\}
 & = b_{ij} (x/\e) \frac{\partial^2 u_0}{\partial x_i \partial x_j}\\
&=-\e \frac{\partial}{\partial x_i}
\left\{ \phi_{kij} (x/\e) \frac{\partial^2 u_0}{\partial x_k \partial x_j} \right\},
\endaligned
$$
where we have used (\ref{b-phi}) for the last step.
Finally, in view of (\ref{L-phi}), we have
$$
\aligned
 & \frac{\partial}{\partial x_i}
\left\{ a_{ij} (x/\e) \frac{\partial}{\partial x_j}
\Big[ \Phi_{\e, k} -x_k -\e \chi_k (x/\e) \Big] \frac{\partial u_0}{\partial x_k} \right\}\\
& \qquad
=
 a_{ij} (x/\e) \frac{\partial}{\partial x_j}
\Big[ \Phi_{\e, k} -x_k -\e \chi_k (x/\e) \Big] \frac{\partial^2 u_0}{\partial x_i \partial x_k}.
\endaligned
$$
This completes the proof.
\end{proof}

We end this section with well known energy estimates for
the initial-Dirichlet problem,

\begin{equation}\label{BVP-1}
\left\{
\aligned
 &(\partial_t^2 +\mathcal{L}_0 )u_0   =0  \quad  \text{ in } \Omega_T,\\
 &  u_0 =0  \quad \text{ on }  S_T,\\
  & u_0 (x, 0)  =\varphi_0 (x), \quad \partial_t u_0 (x, 0)=\varphi_1 (x) \quad \text{ for } x\in \Omega.
  \endaligned
  \right.
\end{equation}
Let $\Omega$ be a bounded domain  in $\mathbb{R}^d$.
Given $\varphi\in H_0^1(\Omega)$ and $\varphi_1 \in L^2(\Omega)$,
the evolution problem (\ref{BVP-1}) has a unique solution in $u_0\in L^\infty(0, T; H_0^1(\Omega))$ with
$\partial_t  u_0 \in L^\infty (0, T; L^2(\Omega))$. 
Moreover, the solution
satisfies 
\begin{equation}\label{e-21}
\|\nabla u_0 (\cdot, t)\|_{L^2(\Omega)}
+ \|\partial_t u_0 (\cdot, t)\|_{L^2(\Omega)}
\le C \Big\{
\|\nabla \varphi_0\|_{L^2(\Omega)}
+ \| \varphi_1 \|_{L^2(\Omega)} \Big\}
\end{equation}
for any $t\in (0, T]$,
where $C$ depends only on $d$ and $\mu$.
Let $\{ \lambda_{0, k}, k=1, 2, \dots \}$ denote the sequence of eigenvalues for $\mathcal{L}_0$ in 
$\Omega$ in an increasing order.
Let $\{ \psi_{0, k} \}$ be a set of eigenfunctions in $H^1_0(\Omega)$ for $\mathcal{L}_0$ in
$\Omega$ such that $\{ \psi_{0, k} \}$ forms  an orthonormal basis for $L^2(\Omega)$
and $\mathcal{L}_{\e, 0} (\psi_{0, k}) =\lambda_{0, k} \psi_{0, k}$ in $\Omega$.
Suppose that
$$
\varphi_{0} =\sum_k a_k \psi_{0, k} 
\quad
\text{ and }
\quad
\varphi_1 =\sum_k b_k \psi_{0, k},
$$
where $a_k, b_k \in \mathbb{R}$.
Then the solution of (\ref{BVP-1}) is given 
\begin{equation}\label{exp}
u_0 (x, t)
=\sum_k 
\Big\{
a_k \cos (\sqrt{\lambda_{0, k}} t)
+ b_k \lambda_{0, k}^{-1/2}
\sin(\sqrt{\lambda_{0, k}} t) \Big\} \psi_{0, k} (x).
\end{equation}
It follows that
\begin{equation}\label{e-220}
\aligned
 & \|\mathcal{L}_0( u_0)  (\cdot, t)\|_{L^2(\Omega)}
+ \|\partial_t \nabla u_0 (\cdot, t)\|_{L^2(\Omega)}
+\|\partial^2_t  u_0(\cdot, t) \|_{L^2(\Omega)}\\
& \qquad\qquad 
\le C \Big\{
\|\mathcal{L}_0 ( \varphi_0) \|_{L^2(\Omega)}
+ \| \nabla \varphi_1 \|_{L^2(\Omega)} \Big\}
\endaligned
\end{equation}
for any $t\in (0, T]$, where $C$ depends only on $d$ and $\mu$.

If $\Omega$ is a bounded $C^{1, 1}$ domain,
 $\varphi_0 \in H^2(\Omega)\cap H_0^1(\Omega)$ and  $\varphi_1\in H_0^1(\Omega)$,
one may use the $H^2$ estimate for the elliptic operator $\mathcal{L}_0$,
$$
\|\nabla^2 u  \|_{L^2(\Omega)} \le C \|\mathcal{L}_0 (u)\|_{L^2(\Omega)}\quad
\text{  for } 
 u\in H_0^1(\Omega) \cap H^2(\Omega),
 $$
 and (\ref{e-220})  to show that
\begin{equation}\label{e-22}
  \|\nabla^2 u_0 (\cdot, t)\|_{L^2(\Omega)}
\le C \Big\{
\|\mathcal{L}_0 ( \varphi_0) \|_{L^2(\Omega)}
+ \| \nabla \varphi_1 \|_{L^2(\Omega)} \Big\}
\end{equation}
for any $t\in (0, T]$, where $C$ depends only on $d$, $\mu$, and $\Omega$.
Furthermore, if $\Omega$ is a bounded $C^3$ domain, 
$\varphi_0 \in H^3(\Omega)\cap H_0^1(\Omega)$ and
$\varphi_1\in H^2(\Omega) \cap H_0^1(\Omega)$, we have 
\begin{equation}\label{e-23}
\aligned
 & \|\nabla^3 u_0 (\cdot, t)\|_{L^2(\Omega)}
+ \|\partial_t \nabla^2  u_0 (\cdot, t)\|_{L^2(\Omega)}
+\|\partial^2_t  \nabla u_0(\cdot, t) \|_{L^2(\Omega)}
+\|\partial_t^3 u_0(\cdot, t) \|_{L^2(\Omega)}  \\
& \qquad\qquad 
\le C \Big\{
\|\mathcal{L}_0 ( \varphi_0) \|_{H^1(\Omega)}
+ \| \mathcal{L}_0 ( \varphi_1)  \|_{L^2(\Omega)} \Big\}
\endaligned
\end{equation}
for any $t\in (0, T]$.



\section{Convergence rates}\label{section-3}

Throughout this section we assume that 
$A=A(y)$ satisfies (\ref{ellipticity}), (\ref{symmetry}) and (\ref{periodicity}).
No additional smoothness condition on $A$ is needed.

For a function $w$ in $\Omega \times [T_0, T_1]$, 
we introduce  the energy  functional,
\begin{equation}\label{energy}
E_\e (t; w)
=\frac 12
 \int_\Omega 
 \Big\{
 \langle A(x/\e)\nabla w (x, t),  \nabla w  (x, t) \rangle
 + (\partial_t w  (x, t) )^2 \Big\} dx
 \end{equation}
 for $t\in  [T_0, T_1]$.
 
\begin{lemma}\label{lemma-2.3}
Let $u_\e$, $u_0$, and $w_\e$ be the same as in Lemma \ref{lemma-2.2}.
Also assume that $u_\e=u_0$ on $\partial\Omega \times [T_0, T_1]$.
 Then
 \begin{equation}\label{e-1}
 \aligned
&  | E_\e (T_1; w_\e)- E_\e (T_0; w_\e)|\\
& \le C \e
 \left(\int_{T_0}^{T_1}\! \!\!\int_\Omega
 \big( |\partial_t \nabla^2 u_0|
 +|\partial_t^2 \nabla u_0| \big)^2\, dx dt \right)^{1/2} \left(\int_{T_0}^{T_1} E_\e (t; w_\e)\, dt \right)^{1/2}\\
 &\qquad
+ C\e \|\nabla^2 u_0 (\cdot, T_1)\|_{L^2 (\Omega)} 
E_\e (T_1; w_\e)^{1/2}\\
&\qquad
+ C\e \|\nabla^2 u_0 (\cdot, T_0)\|_{L^2 (\Omega)} 
E_\e ( T_0; w_\e)^{1/2},
 \endaligned
 \end{equation}
 where $C$ depends only on $d$ and $\mu$.
\end{lemma}

\begin{proof}
Using the symmetry condition (\ref{symmetry}), we obtain 
\begin{equation}\label{2.3-1}
E_\e ( T_1; w_\e)-E_\e (T_0; w_\e)
 =\int_{T_0}^{T_1}\!\!\! 
 \langle
\left(\partial_t^2 +\mathcal{L}_\e\right) w_\e, \partial_t w_\e \rangle_{H^{-1}(\Omega) \times H^1_0(\Omega)}\,
 dt.\\
\end{equation}
We will use the formula (\ref{L-w}) for 
$(\partial_t^2 +\mathcal{L}_\e) w_\e $ 
 to bound the right-hand side of (\ref{2.3-1}).
 The fact $w_\e=0$ on $\partial \Omega \times [T_0, T_1]$ is also used.

Let $I_1$ denote the first term in the right-hand side of (\ref{L-w}).
It follows from integration by parts (first in $x$ and then in $t$) that
$$
\aligned
\Big| 
\int_{T_0}^{T_1}\! \!\! \langle 
I_1,  \partial_t w_\e \rangle_{H^{-1}(\Omega) \times H_0^1(\Omega)}  \, dt \Big|
&=\e \Big| \int_{T_0}^{T_1}\! \!\!\int_\Omega
\phi_{kij} (x/\e) \frac{\partial^2 u_0}{\partial x_k \partial x_j} \cdot \partial_t \frac{\partial w_\e}{\partial x_i}\, dx dt \Big|\\
& \le C \e  \int_{T_0}^{T_1}\! \!\!\int_\Omega
|\partial_t \nabla^2 u_0|\, |\nabla w_\e|\, dx dt\\
&\qquad + C \e \int_\Omega |\nabla^2 u_0 (x, T_1)|\, |\nabla w_\e (x, T_1)|\, dx\\
&\qquad+C \e \int_\Omega |\nabla^2 u_0 (x, T_0)|\, |\nabla w_\e (x, T_0)|\, dx.
\endaligned
$$
By the Cauchy inequality this leads to
\begin{equation}\label{2.3-2}
\aligned
\Big| \int_{T_0}^{T_1}\! \!\! 
\langle I_1,  \partial_t w_\e \rangle_{H^{-1}(\Omega) \times H^1_0(\Omega)}\,  dt \Big|
&\le C \e \|\partial_t \nabla^2 u_0\|_{L^2(\Omega\times (T_0, T_1))}
\left(\int_{T_0}^{T_1} E_\e (t; w_\e)\, dt \right)^{1/2}\\
&\qquad
+ C\e \|\nabla^2 u_0 (\cdot, T_1)\|_{L^2 (\Omega)} 
E_\e (T_1; w_\e)^{1/2}\\
&\qquad
+ C\e \|\nabla^2 u_0 (\cdot, T_0)\|_{L^2 (\Omega)} 
E_\e (T_0; w_\e)^{1/2},
\endaligned
\end{equation}
where $C$ depends only on $d$ and $\mu$.
Let $I_2$ denote the second term in the right-hand side of (\ref{L-w}).
Since $\| \Phi_{\e, k} -x_k \|_{L^\infty(\Omega)}
\le C \e$, it is easy to see that (\ref{2.3-2}) also holds with $I_2$ in the place of $I_1$.

Next, let $I_3$ denote the third term in the right-hand side of (\ref{L-w}).
Using  integration by parts in the $t$ variable, we see that
$$
\aligned
\Big| \int_{T_0}^{T_1}\! \!\! \int_\Omega
I_3 \cdot \partial_t w_\e\, dx dt \Big|
&\le 
C \int_{T_0}^{T_1}\!\!\!\int_\Omega
|\nabla \big[ \Phi_\e -x -\e \chi(x/\e)\big] | \, |\partial_t \nabla^2 u_0|\, | w_\e|\, dx dt\\
&\qquad
+C \int_\Omega |\nabla \big[  \Phi_\e - x -\e \chi(x/\e) \big] |\,  |\nabla^2 u_0 (x, T_1)|\, |w_\e(x, T_1)|\, dx\\
&\qquad
+C \int_\Omega |\nabla \big[ \Phi_\e - x -\e \chi(x/\e) \big] |\,  |\nabla^2 u_0 (x, T_0)|\, |w_\e(x, T_0)|\, dx.
\endaligned
$$
It follows from the Cauchy inequality that
$$
\aligned
\Big| \int_{T_0}^{T_1}\! \!\! \int_\Omega
I_3 \cdot \partial_t w_\e\, dx dt \Big|
&\le C \| \nabla \big[ \Phi_\e -x -\e \chi(x/\e)\big] w_\e\|_{L^2(\Omega\times (T_0, T_1))}
\|\partial_t \nabla^2 u_0\|_{L^2(\Omega \times (T_0, T_1))}\\
&\qquad +  C\| \nabla \big[ \Phi_\e -x -\e \chi(x/\e)\big] w_\e (\cdot, T_1) \|_{L^2(\Omega)}
\|\nabla^2 u_0 (\cdot, T_1)\|_{L^2(\Omega)}\\
&\qquad +   C \| \nabla \big[ \Phi_\e -x -\e \chi(x/\e)\big] w_\e (\cdot, T_0) \|_{L^2(\Omega)}
\|\nabla^2 u_0 (\cdot, T_0)\|_{L^2(\Omega)}.
\endaligned
$$
Since $\mathcal{L}_\e (\Phi_\e -x -\e \chi(x/\e))=0$ in $\Omega$ and
$w_\e=0$ on $\partial\Omega$,
by Caccioppoli's inequality, we have 
\begin{equation}\label{Ca}
\aligned
\| \nabla \big[ \Phi_\e -x -\e \chi(x/\e)\big] w_\e (\cdot, t) \|_{L^2(\Omega)}
&\le C \| \big[\Phi_\e -x -\e \chi(x/\e)\big] \nabla w_\e (\cdot, t) \|_{L^2(\Omega)}\\
&\le C \e \|\nabla w_\e (\cdot, t)\|_{L^2(\Omega)}
\endaligned
\end{equation}
for $t\in [T_0, T_1]$. 
As a result, the estimate (\ref{2.3-2}) continues to hold if we replace $I_1$ by $I_3$.
 
Finally, let $I_4$ denote the last term in the right-hand side of (\ref{L-w}).
By the Cauchy inequality, we obtain 
$$
\aligned
\Big| \int_{T_0}^{T_1}\! \!\! \int_\Omega
I_4 \cdot \partial_t w_\e\, dx dt \Big|
&\le C\e \|\partial_t^2 \nabla u_0\|_{L^2(\Omega\times (T_0, T_1))}
\|\partial_t w_\e \|_{L^2(\Omega\times (T_0, T_1))}\\
&\le C \e \|\partial_t^2 \nabla u_0\|_{L^2(\Omega\times (T_0, T_1))}
\left(\int_{T_0}^{T_1} E_\e (t; w_\e)\, dt \right)^{1/2}.
\endaligned
$$
This completes the proof of (\ref{e-1}).
\end{proof}

The next lemma gives an estimate of $E_\e (t; w_\e)$ for $t=0$.

\begin{lemma}\label{lemma-2.4}
Let  $w_\e$, $\varphi_{ \e, 0}$, $\varphi_0$, $\varphi_{ \e, 1}$ and $\varphi_1$ be the same as in Theorem \ref{main-theorem-1}.
Then
\begin{equation}\label{2.4-0}
\aligned
E_\e (0; w_\e)
&\le C \|\mathcal{L}_\e  (\varphi_{\e, 0}) -\mathcal{L}_0 (\varphi_0) \|^2_{H^{-1}(\Omega)}
+ C \|\varphi_{\e, 1} -\varphi_1\|^2_{L^2(\Omega)}\\
&\qquad
+  C \e^2 \Big\{ 
\|\nabla^2 \varphi_0\|_{L^2(\Omega)}^2
+ \|\nabla \varphi_1\|^2_{L^2(\Omega)} \Big\},
\endaligned
\end{equation}
where $C$ depends only on $d$ and $\mu$.
\end{lemma}

\begin{proof}
Note that
$$
\aligned
\partial_t w_\e (x, 0) & =\partial_t u_\e (x, 0)- \partial _t u_0 (x, 0)- \big( \Phi_{\e, k} -x_k\big) \partial_t \frac{\partial u_0}{\partial x_k} (x, 0)\\
&= \varphi_{\e, 1} -\varphi_1 - \big( \Phi_{\e, k} - x_k \big) \frac{\partial \varphi_1}{\partial x_k}.
\endaligned
$$
It follows that
$$
\| \partial_t w_\e (\cdot, 0)\|_{L^2(\Omega)}
\le  \| \varphi_{\e, 1}-\varphi_1\|_{L^2(\Omega)}
+ C \e \|\nabla \varphi_1 \|_{L^2(\Omega)}.
$$
Next, to bound $\|\nabla w_\e (\cdot, 0)\|_{L^2(\Omega)}$,
we use
\begin{equation}\label{2.4-4}
\int_\Omega \langle A(x/\e)\nabla w_\e, \nabla w_\e \rangle \, dx
=\int_\Omega \langle \mathcal{L}_\e (w_\e),   w_\e\rangle_{H^{-1}(\Omega) \times H^1_0(\Omega)}
\, dx
\end{equation}
and the following  formula,
\begin{equation}\label{L-w-1}
\aligned
\mathcal{L}_\e  (w_\e) (x, 0)
&= \mathcal{L}_\e (\varphi_{\e, 0} )
-\mathcal{L}_0 (\varphi_0)
-\e \frac{\partial}{\partial x_i}
\left\{ \phi_{kij} (x/\e) \frac{\partial^2 \varphi_0}{\partial x_k\partial x_j} \right\}\\
&\qquad
+ \frac{\partial}{\partial x_i}
\left\{ a_{ij}(x/\e) \big[ \Phi_{\e, k} -x_k\big] \frac{\partial^2\varphi_0}{\partial x_j \partial x_k} \right\}\\
&\qquad
+a_{ij}(x/\e) 
\frac{\partial}{\partial x_j}
\big[ \Phi_{\e, k} -x_k -\e \chi_k (x/\e) \big] \frac{\partial^2\varphi_0}{\partial x_i \partial x_k}.
\endaligned
\end{equation}
The proof of (\ref{L-w-1}) is similar to that of (\ref{L-w}).
It follows from (\ref{2.4-4}) and (\ref{L-w-1}) that
$$
\aligned
\| \nabla w_\e (\cdot, 0)\|_{L^2(\Omega)}^2
&\le  C \| \mathcal{L}_\e (\varphi_{\e, 0}) -\mathcal{L}_0 (\varphi_0)\|_{H^{-1} (\Omega)}  \| \nabla w_\e (\cdot, 0)\|_{L^2(\Omega)}\\
&\qquad + C \e \|\nabla^2 \varphi_0\|_{L^2(\Omega)}
\|\nabla w_\e (\cdot, 0)\|_{L^2(\Omega)}\\
& \qquad
+ C\|\nabla [ \Phi_\e - x -\e \chi(x/\e)] w_\e (\cdot, 0) \|_{L^2 (\Omega)} 
\| \nabla^2 \varphi_0\|_{L^2(\Omega)}\\
&\le
C \| \mathcal{L}_\e (\varphi_{\e, 0}) -\mathcal{L}_0 (\varphi_0)\|_{H^{-1} (\Omega)}  \| \nabla w_\e (\cdot, 0)\|_{L^2(\Omega)}\\
&\qquad + C \e \|\nabla^2 \varphi_0\|_{L^2(\Omega)}
\|\nabla w_\e (\cdot, 0)\|_{L^2(\Omega)},
\endaligned
$$
where we have used the Caccioppoli's inequality (\ref{Ca})  for the last step.
This yields 
$$
\|\nabla w_\e (\cdot, 0)\|_{L^2(\Omega)}
\le C  \| \mathcal{L}_\e (\varphi_{\e, 0})  -\mathcal{L}_0 (\varphi_0)\|_{H^{-1} (\Omega)}  
+C \e \|\nabla^2 \varphi_0\|_{L^2(\Omega)}
$$
and completes the proof.
\end{proof}

We are now in a position to give the proof of Theorem \ref{main-theorem-1}

\begin{proof}[\bf Proof of Theorem \ref{main-theorem-1}]
Let
\begin{equation}\label{M-M}
\aligned
M_0  &=\sup_{0\le t\le T}
\left(\int_\Omega |\nabla^ 2 u_0 (x, t)|^2 dx \right)^{1/2},\\
M_1& =\sup_{0\le t\le T}
\left(\int_\Omega
\big( |\partial_t \nabla^2 u_0 (x, t)|
+|\partial_t^2 \nabla u_0(x, t) | \big)^2 dx \right)^{1/2}.
\endaligned
\end{equation}
Let $w_\e$ be defined by (\ref{w}).
We will show that for any $t\in [0, T]$,
\begin{equation}\label{m-e}
E_\e (t; w_\e)
\le C \Big\{ E_\e (0; w_\e)
+ \e^2 M_0^2  + \e^2 T^2 M^2_1 \Big\},
\end{equation}
where $C$ depends only on $d$ and $\mu$.
This, together with the estimate of  $E_\e (0; w_\e)$ in Lemma \ref{lemma-2.4},
gives the inequality (\ref{main-estimate-1}).

It follows by Lemma \ref{lemma-2.3} that for $0\le t\le T$,
$$
E_\e (t; w_\e)
\le E_\e (0; w_\e)
+ C \e (  T  M_1 + M_0) \sup_{t\in [0, T]}
E_\e (t; w_\e)^{1/2},
$$
where $C$ depends only on $d$ and $\mu$.
This yields 
$$
\aligned
\sup_{t\in [0, T]} E_\e (t; w_\e)
& \le E(0; w_\e)
+ C \e (TM_1+M_0)  \sup_{t\in [0, T]}
E_\e ( t; w_\e)^{1/2} \\
 & \le E_\e (0; w_\e)
+ C \e^2(T^2M^2_1+ M_0^2) 
+\frac12 \sup_{t\in [0, T]} E_\e (t; w_\e),
\endaligned
$$
from which the estimate (\ref{m-e}) follows.
\end{proof}

We end this section with a convergence rate for $\|u_\e (\cdot, t) -u_0 (\cdot, t) \|_{L^2(\Omega)}$
for $0<t<T$.
Consider the initial-Dirichlet problem,
\begin{equation}\label{BVP-3}
\left\{
\aligned
& \big( \partial_t^2 +\mathcal{L}_\e \big) u_\e =0 \quad \text{ in } \Omega_T=\Omega\times (0, T] ,\\
& u_\e=0 \quad \text{ on } S_T=\partial\Omega \times [0, T],\\
&  u_\e (x, 0)=\varphi_0 (x), \quad \partial_t u_\e (x, 0)=\varphi_1 (x) \quad \text{ for } x\in \Omega,
\endaligned
\right.
\end{equation}
and its homogenized problem,
\begin{equation}\label{BVP-3-0}
\left\{
\aligned
& \big( \partial_t^2 +\mathcal{L}_0 \big) u_0 =0 \quad \text{ in } \Omega_T ,\\
& u_0=0 \quad \text{ on } S_T ,\\
&  u_0 (x, 0)=\varphi_0 (x), \quad \partial_t u_0 (x, 0)=\varphi_1 (x) \quad \text{ for } x\in \Omega,
\endaligned
\right.
\end{equation}
where $\varphi_0\in H_0^1(\Omega) \cap H^2(\Omega)$ and $\varphi_1 \in H^1_0(\Omega)$.
Let
$$
v_\e (x, t)=\int_0^t u_\e (x, s)\, ds 
\quad \text{ and } \quad
v_0 (x, t)=\int_0^t u_0 (x, s)\, ds.
$$
Then
$$
\left\{
\aligned
& \big( \partial_t^2 +\mathcal{L}_\e \big) v_\e = \varphi_1\quad \text{ in } \Omega_T ,\\
& v_\e=0 \quad \text{ on } S_T,\\
&  v_\e (x, 0)=0,  \quad \partial_t v_\e (x, 0)=\varphi_0 (x) \quad \text{ for } x\in \Omega,
\endaligned
\right.
$$
and
$$
\left\{
\aligned
& \big( \partial_t^2 +\mathcal{L}_0 \big) v_0 = \varphi_1\quad \text{ in } \Omega_T ,\\
& v_0=0 \quad \text{ on } S_T,\\
&  v_0 (x, 0)=0,  \quad \partial_t v_0 (x, 0)=\varphi_0 (x) \quad \text{ for } x\in \Omega.
\endaligned
\right.
$$
By applying  Theorem \ref{main-theorem-1}  to $v_\e$ and $v_0$ and using (\ref{u-t}),
we see  that for any $t\in (0, T]$,
\begin{equation}\label{3-40}
\aligned
& \| u_\e (\cdot, t)-u_0 (\cdot, t)\|_{L^2(\Omega)}\\
&  \le C \e \|\nabla \varphi_0\|_{L^2(\Omega)}
 +  C \e \sup_{t\in [0, T]}
\|\nabla^2 v_0 (\cdot, t) \|_{L^2(\Omega)}
+C \e \sup_{t\in [0, T]}
\|\nabla u_0 (\cdot, t) \|_{L^2(\Omega)}
\\
& \qquad
+ C \e T 
\sup_{t\in (0, T)} 
\| |\nabla^2 u_0 (\cdot, t)| + | \partial_t \nabla u_0 | \|_{L^2(\Omega)},
\endaligned
\end{equation}
where we have used the fact $\partial_t v_0 =u_0$.
Note that, if $\Omega$ is $C^{1, 1}$,
$$
\aligned
\|\nabla^2 v_0 (\cdot, t) \|_{L^2(\Omega)}
&\le C \|\mathcal{L}_0 (v_0) (\cdot, t) \|_{L^2(\Omega)}\\
&\le C \|\partial_t u_0 (\cdot, t)\|_{L^2(\Omega)}
+ C \| \varphi_1\|_{L^2(\Omega)}\\
&\le C \big\{  \|\nabla \varphi_0\|_{L^2(\Omega)} + \|\varphi_1\|_{L^2(\Omega)} \big\},
\endaligned
$$
where we have used (\ref{e-21}) for the last inequality.
This, together with (\ref{3-40}), (\ref{e-220}) and (\ref{e-22}),  yields that
\begin{equation}\label{3-41}
\aligned
& \| u_\e (\cdot, t) -u_0 (\cdot, t)\|_{L^2(\Omega)}\\
&\le C \e \big\{ \|\nabla \varphi_0\|_{L^2(\Omega) }+ \| \varphi_1\|_{L^2(\Omega) } \big\}\\
& \qquad
+  C \e T
\left( \| \varphi_0\|_{H^2(\Omega)} 
+ \|  \varphi_1\|_{H^1(\Omega)} \right)
\endaligned
\end{equation}
for any $t \in (0, T]$, where $\Omega$ is $C^{1,1}$ and the constant $C$ depends only on $d$, $\mu$ and $\Omega$.



\section{Uniform boundary controllability}\label{section-4}

Throughout this section we will assume that $A=A(y)$ satisfies conditions (\ref{ellipticity}),
(\ref{symmetry}) and  (\ref{periodicity}) as well as  the Lipschitz condition (\ref{Lip-C}).

Let $u_\e$ be the solution of the initial-Dirichlet problem,
\begin{equation}\label{BVP-4}
\left\{
\aligned
& \big( \partial_t^2 +\mathcal{L}_\e \big) u_\e =0 \quad \text{ in } \Omega_T=\Omega\times (0, T] ,\\
& u_\e=0 \quad \text{ on } S_T= \partial\Omega \times [0, T],\\
&  u_\e (x, 0)=\varphi_{\e, 0} (x), \quad \partial_t u_\e (x, 0)=\varphi_{\e, 1} (x) \quad \text{ for } x\in \Omega.
\endaligned
\right.
\end{equation}
We are interested in the estimates (\ref{upper-b}) and (\ref{lower-b}) 
with positive constants $C$ and $c$ independent of $\e>0$.

Let $h=(h_1, h_2, \dots, h_d)$ be a vector field in $ C^1(\mathbb{R}^d; \mathbb{R}^d)$ and $n=(n_1, n_2, \dots, n_d)$
 denote the outward unit normal to $\partial\Omega$.
We start with the following well known Rellich identity,
\begin{equation}\label{R-0}
\aligned
\int_0^T\!\!\!\int_{\partial\Omega} \langle h,  n\rangle  \cdot a_{ij}^\e  \frac{\partial u_\e}{\partial x_j} \cdot \frac{\partial u_\e}{\partial x_i} \, d\sigma dt
= & 
2\int_0^T\!\!\! \int_{\partial\Omega}
h_k a_{ij}^\e  \frac{\partial u_\e}{\partial x_i} \left\{ n_k \frac{\partial u_\e}{\partial x_j}
-n_j \frac{\partial u_\e}{\partial x_k} \right\} d\sigma  dt\\
& \qquad
-\int_0^T\!\!\! \int_\Omega \text{\rm div} (h) \cdot 
 a_{ij}^\e  \frac{\partial u_\e}{\partial x_j} \cdot \frac{\partial u_\e}{\partial x_i} \, dx dt \\
&\qquad
  -\int_0^T\!\!\! \int_\Omega h_k \frac{\partial a_{ij}^\e } {\partial x_k} \cdot  \frac{\partial u_\e}{\partial x_j}\cdot  \frac{\partial u_\e}{\partial x_i}\, dx dt \\
 & \qquad
 +  2\int_0^T\!\!\!  \int_{\Omega}  \frac{\partial h_k}{\partial x_j} \cdot 
 a_{ij}^\e  \frac{\partial u_\e}{\partial x_k} \cdot \frac{\partial u_\e}{\partial x_i} \, dx dt \\
 &\qquad
  - 2\int_0^T\!\!\!  \int_\Omega h_k \frac{\partial u_\e}{\partial x_k} \cdot \mathcal{L}_\e (u_\e)\, dx dt ,
 \endaligned
 \end{equation}
where $a^\e_{ij}= a_{ij}(x/\e)$.
The identity (\ref{R-0})  follows from integration by parts (in the $x$ variable).
We remark that the symmetry condition (\ref{symmetry}), which is essential for (\ref{R-0}) even in the case of constant coefficients, is used to obtain 
$$
 a_{ij}^\e \frac{\partial}{\partial x_k}
\left( \frac{\partial u_\e}{\partial x_i} \cdot \frac{\partial u_\e}{\partial x_j} \right)
=2 \mathcal{L}_\e (u_\e) \cdot \frac{\partial u_\e}{\partial x_k}
+
2 \frac{\partial}{\partial x_i} 
\left( a_{ij} ^\e \frac{\partial u_\e}{\partial x_j} \cdot \frac{\partial u_\e}{\partial x_k} \right)
$$
in the proof of (\ref{R-0}).
It  also follows from integration by parts   that
\begin{equation}\label{R-1}
\aligned
\int_0^T\!\!\! \int_\Omega
h_k \frac{\partial u_\e}{\partial x_k} \cdot \partial_t^2 u_\e \, dx dt 
&=-\frac 12 \int_0^T \!\!\! \int_{\partial\Omega} 
\langle h, n \rangle  \cdot (\partial_t u_\e)^2 \, d\sigma dt\\
&\qquad 
+\frac12 \int_0^T\!\!\! \int_\Omega \text{\rm div} (h) \cdot (\partial_t u_\e )^2\, dx dt\\
&\qquad
+ \int_\Omega h_k \frac{\partial u_\e}{\partial x_k} (x, T) \partial_t u_\e (x, T) \, dx \\
&\qquad
- \int_\Omega h_k \frac{\partial u_\e}{\partial x_k} (x, 0) \partial_t u_\e (x, 0) \, dx.
\endaligned
\end{equation}
Suppose $u_\e=0$ on $\partial\Omega$.
Since $n_k \frac{\partial u_\e }{\partial x_j} -n_j \frac{\partial u_\e}{\partial x_k} =0$ and $\partial_t u_\e =0$ on $\partial\Omega$,
by combining (\ref{R-0}) with (\ref{R-1}), we obtain 
\begin{equation}\label{R-2}
\aligned
\int_0^T\!\!\!\int_{\partial\Omega} \langle h, n \rangle 
 \cdot a_{ij}^\e  \frac{\partial u_\e}{\partial x_j} \cdot \frac{\partial u_\e}{\partial x_i} \, d\sigma dt
= & \int_0^T\!\!\! \int_\Omega \text{\rm div} (h) \cdot 
\Big ( (\partial_t u_\e)^2- a_{ij}^\e  \frac{\partial u_\e}{\partial x_j}   \cdot \frac{\partial u_\e}{\partial x_i}\Big)  dx dt \\
&\qquad
  -\int_0^T\!\!\! \int_\Omega h_k \frac{\partial a_{ij}^\e } {\partial x_k} \cdot  \frac{\partial u_\e}{\partial x_j}\cdot  \frac{\partial u_\e}{\partial x_i}\, dx dt \\
 & \qquad
 +  2\int_0^T\!\!\!  \int_{\Omega}  \frac{\partial h_k}{\partial x_j} \cdot 
 a_{ij}^\e  \frac{\partial u_\e}{\partial x_k} \cdot \frac{\partial u_\e}{\partial x_i} \, dx dt \\
 &\qquad
  - 2\int_0^T\!\!\!  \int_\Omega h_k \frac{\partial u_\e}{\partial x_k} \cdot \big( \partial_t^2 + \mathcal{L}_\e \big)u_\e \, dx dt\\
  &\qquad
+ \int_\Omega h_k \frac{\partial u_\e}{\partial x_k} (x, T) \partial_t u_\e (x, T) \, dx \\
&\qquad
- \int_\Omega h_k \frac{\partial u_\e}{\partial x_k} (x, 0) \partial_t u_\e (x, 0) \, dx.
 \endaligned
 \end{equation}

\begin{lemma}\label{lemma-4.1}
Let $\Omega$ be a bounded Lipschitz domain in $\mathbb{R}^d$.
Let $u_0$ be a weak solution of \eqref{BVP-3-0} for the homogenized operator $\partial_t^2+ \mathcal{L}_0$. Then
\begin{equation}\label{upper-b-0}
\int_0^T\!\!\!\int_{\partial\Omega}
|\nabla u_0|^2\, d\sigma dt
\le C  (T r_0^{-1} +1) \Big\{ \| \nabla \varphi_0 \|_{L^2(\Omega)}^2
+ \|\varphi_1 \|_{L^2(\Omega)}^2 \Big\},
\end{equation}
where  $r_0$ denotes the diameter of $\Omega$.
Moreover, if $T \ge C_0r_0 $, 
\begin{equation}\label{lower-b-0}
T r_0^{-1}   \Big\{ \|\nabla \varphi_0 \|_{L^2(\Omega)}^2
+ \|\varphi_1 \|_{L^2(\Omega)}^2 \Big\}
\le C  \int_0^T\!\!\!\int_{\partial\Omega}
|\nabla u_0|^2\, d\sigma dt.
\end{equation}
The constants  $C$  and $C_0$ depend only on $d$, $\mu$ and the Lipschitz character of $\Omega$.
\end{lemma}

\begin{proof}
This is well known and follows readily from (\ref{R-2}) (with $\widehat{a}_{ij}$ in the place of $a_{ij}^\e$)
(see e.g. \cite{Lions-1988}).
We include  a proof here for  the reader's convenience.
To  see (\ref{upper-b-0}),
we choose a vector field $h\in C^1(\mathbb{R}^d; \mathbb{R}^d)$ such that
$\langle h, n \rangle\ge c_0>0$ on $\partial\Omega$ and $|\nabla h|\le C/r_0$.
It follows from (\ref{R-2}), with $\widehat{a}_{ij}$ in the place of $a_{ij}^\e$,  that
$$
\aligned
c \int_0^T\!\!\! \int_{\partial\Omega}
|\nabla u_0|^2 \, d\sigma dt
 & \le \frac{C}{r_0} \int_0^T\!\!\!\int_\Omega 
\big( |\nabla u_0|^2 +|\partial_t u_0|^2 \big)\, dx dt\\
& \qquad
+ C \int_\Omega |\nabla u_0 (x, T)| \, |\partial_t u_0(x, T)|\, dx
+ C \int_\Omega |\nabla \varphi_0 |\, |\varphi_1|\, dx\\
&\qquad
\le C ( T r_0^{-1} +1) 
\Big\{ \|\nabla \varphi_0\|_{L^2(\Omega)}^2
+\|\varphi_1\|_{L^2(\Omega)}^2 \Big\},
\endaligned
$$
where we have used the energy estimate (\ref{e-21})  for the last step.

To prove (\ref{lower-b-0}),
we choose $h (x)=x-x_0$, where $x_0\in \Omega$.
Note that $\text{\rm div} (h)=d$.
It follows from (\ref{R-2}) that 
\begin{equation}\label{lower-1}
\aligned
& \Big|
-\int_0^T\!\!\! \int_{\partial\Omega}
\langle h, n \rangle\cdot  \widehat{a}_{ij} \frac{\partial u_0}{\partial x_j} \cdot
\frac{\partial u_0}{\partial x_i} \, d\sigma dt
+ d X
+ (2-d)  Y  \Big|\\
& \qquad
\le C r_0 
\Big\{ \|\nabla \varphi_0\|_{L^2(\Omega)}^2
+\|\varphi_1\|_{L^2(\Omega)}^2 \Big\},
\endaligned
\end{equation}
where
$$
\aligned
X&= \int_0^T\!\!\! \int_\Omega (\partial_t u_0)^2\, dx dt,\\
Y&=\int_0^T\!\!\! \int_\Omega \widehat{a}_{ij} \frac{\partial u_0}{\partial x_j} \cdot
\frac{\partial u_0}{\partial x_i} \, dx  dt.
\endaligned
$$
Note that by the  conservation of energy, 
$$
X+Y = T \int_\Omega 
\left( \varphi_1^2 + \widehat{a}_{ij} \frac{\partial \varphi_0}{\partial x_j} \cdot \frac{\partial \varphi_0}{\partial x_i}\right)  dx,
$$
and that 
$$
\aligned
X-Y & =\int_0^T \!\!\! \int_\Omega \partial_t ( u_0 \partial_t u_0 ) \, dx dt\\
&=\int_\Omega u_0 (x, T) \partial_t u_0 (x, T)\, dx- 
\int_\Omega \varphi_0 \varphi_1\, dx\\
&\le C r_0 \Big\{ \|\nabla \varphi_0\|_{L^2(\Omega)}^2
+ \|\varphi_1\|_{L^2(\Omega)}^2 \Big\},
\endaligned
$$
where we have used Poincar\'e's inequality and the energy estimates for the last step.
By writing $dX + (2-d) Y$ as $ (X+Y) + (d-1) (X-Y)$,
we deduce from (\ref{lower-1}) that
$$
\aligned
& \Big|
\int_0^T\!\!\! \int_{\partial \Omega}
\langle h, n \rangle \cdot  \widehat{a}_{ij} \frac{\partial u_0}{\partial x_j} \cdot
\frac{\partial u_0}{\partial x_i} \, d\sigma dt
-
T \int_\Omega 
\left( \varphi_1^2 + \widehat{a}_{ij} \frac{\partial \varphi_0}{\partial x_j} \cdot \frac{\partial \varphi_0}{\partial x_i}\right)  dx\Big|\\
&
\le C r_0 \Big\{ \|\nabla \varphi_0\|_{L^2(\Omega)}^2
+ \|\varphi_1\|_{L^2(\Omega)}^2 \Big\},
\endaligned
$$
from which the inequality  (\ref{lower-b-0}) follows if
$T\ge C_0 r_0$.
\end{proof}

The argument used in the proof of Lemma \ref{lemma-4.1} for $\partial_t^2 +\mathcal{L}_0$
does not  work for the operator $\partial_t^2 + \mathcal{L}_\e$;
 the derivative of $a_{ij}^\e$ is unbounded as $\e\to  0$.
 Our approach to Theorem \ref{main-theorem-2}  is to approximate  the solution $u_\e$ of (\ref{BVP-4}) with initial data 
 $( \varphi_{\e, 0}, 
 \varphi_{\e, 1}) $
  by a solution of (\ref{BVP-3-0}) for
 the homogenized operator $\partial_t^2 +\mathcal{L}_0$
 with initial data  $( \varphi_0, \varphi_1) $,
 where $\varphi_1 =\varphi_{\e, 1}$ and $\varphi_0$ is the function  in $H^1_0 (\Omega)$ such that
\begin{equation}\label{data-0}
\mathcal{L}_ 0 (\varphi_{0}) =\mathcal{L}_\e (\varphi_{\e, 0})  \quad \text{ in }  \Omega.
\end{equation}

\begin{lemma}\label{lemma-4.2}
Let $\Omega$ be a bounded $C^3$ domain in $\mathbb{R}^d$.
Let $u_\e$ and $u_0$ be the solutions of \eqref{BVP-4} and \eqref{BVP-3-0}
with initial data $(\varphi_{\e, 0}, \varphi_{\e, 1}) $ and
$(\varphi_0, \varphi_1)$, respectively.
Assume that $\varphi_1=\varphi_{1, \e}\in H^2(\Omega)\cap H_0^1(\Omega)$ 
and $\varphi_0 \in H^3(\Omega) \cap H_0^1(\Omega)$ satisfies \eqref{data-0}.
Let $w_\e$ be given by \eqref{w}.
Then for $0<\e<\min( r_0, T)$, 
\begin{equation}\label{4.2-0}
\aligned
\int_0^T \!\!\!\int_{\partial\Omega}
|\nabla w_\e|^2\, d\sigma dt
 & \le C T\e \Big\{ \|\varphi_0\|^2_{H^2(\Omega)}
 +\|\varphi_1\|^2_{H^1(\Omega)} \Big\}\\
 & \qquad
 + C T^{3} \e 
 \Big\{ \|\varphi_0 \|^2_{H^3(\Omega)}
 + \|\varphi_1\|^2_{H^2(\Omega)} \Big\}
 \\
 &\qquad
 + CT \e^3 \Big\{ \|\varphi_0\|^2_{H^3(\Omega)}
 + \|\varphi_1\|^2_{H^2(\Omega)}  \Big\},
 \endaligned
\end{equation}
where $C$ depends only on $d$, $\mu$, $M$,  and $\Omega$.
\end{lemma}

\begin{proof}
Let $h$ be  a vector  field in $C^1(\mathbb{R}^d; \mathbb{R}^d)$ such that
$\langle h, n\rangle \ge c_0>0$ on $\partial\Omega$ and $|\nabla h|\le C r_0^{-1}$.
We apply the Rellich identity (\ref{R-2}) with $w_\e$ in the place of $u_\e$.
This gives
\begin{equation}\label{4.2-1}
\aligned
\int_0^T\!\!\!\int_{\partial\Omega}
|\nabla w_\e|^2\, d\sigma dt
&\le \frac{C}{\e} \int_0^T\!\!\!  \int_\Omega |\nabla w_\e|^2\, dx dt
+C \int_0^T\!\!\!\int_\Omega |\nabla w_\e | | (\partial_t^2 +\mathcal{L}_\e) w_\e|\, dx dt\\
&\qquad
+  C \sup_{t\in [0, T]}
\|\nabla w_\e (\cdot, t) \|_{L^2(\Omega)}
\|\partial_t w_\e (\cdot, t) \|_{L^2(\Omega)}\\
&\le
C T\e^{-1} \sup_{ t \in [0, T]}
\Big\{ \|\nabla w_\e (\cdot, t)\|_{L^2(\Omega)}^2
+ \|\partial_t w_\e (\cdot, t) \|_{L^2(\Omega)}^2 \Big\}\\
&\qquad
+ C \e\int_0^T \!\!\!\int_\Omega
|(\partial_t^2 + \mathcal{L}_\e) w_\e |^2\, dx dt,
\endaligned
\end{equation}
where we have used the Cauchy inequality for the last step.
Since $\Omega$ is $C^{3}$ and $A$ is Lipschitz,
$\nabla \Phi_\e$ is bounded.
Also, under the smoothness condition (\ref{Lip-C}), the functions  $\nabla \chi_j$ and $\nabla \phi_{kij}$ are bounded.
Thus, in view of (\ref{L-w}), we obtain 
\begin{equation}\label{4.2-2}
|(\partial_t^2 +\mathcal{L}_\e)  w_\e |
\le C \big\{ |\nabla^2 u_0| +\e |\nabla^3 u_0| + \e |\partial^2_t \nabla u_0| \big\}.
\end{equation}
This, together with  (\ref{4.2-1}) and Theorem \ref{main-theorem-1}, gives
$$
\aligned
& \int_0^T\!\!\!\int_{\partial\Omega}
|\nabla w_\e|^2\, d\sigma dt\\
&
  \le C T \e 
\Big\{ \|  \varphi_0 \|^2_{H^2(\Omega)}
+ \|  \varphi_1 \|^2_{H^1(\Omega)} \Big\} \\
&\qquad
+ C T \e \sup_{t\in (0, T]}  \|  \nabla^2  u_0 (\cdot, t) \|^2_{L^2(\Omega)}\\
&\qquad
 + C T^3  \e  
 \sup_{t\in (0, T]}  \| |\partial_t \nabla^2 u_0 (\cdot, t) | +| \partial_t^2 \nabla u_0 (\cdot, t)| \|^2_{L^2(\Omega)}\\
&\qquad
+ C T \e^3
\sup_{t\in (0, T]}
\| |\nabla ^3 u_0 (\cdot, t)|
+|\partial_t^2 \nabla u_0 (\cdot, t) | \|^2_{L^2(\Omega)}, 
\endaligned
 $$
from which the estimate (\ref{4.2-0}) follows by using the energy estimates (\ref{e-22}) and (\ref{e-23}).
\end{proof}

The next theorem provides an upper bound for $\|\nabla  u_\e\|_{L^2(S_T) }$.

\begin{thm}\label{theorem-4.1}
Assume that $A$ satisfies conditions \eqref{ellipticity}, \eqref{symmetry},
\eqref{periodicity},  and \eqref{Lip-C}.
Let $\Omega$ be a bounded $C^3$ domain in $\mathbb{R}^d$.
Let $u_\e$ be a weak solution of \eqref{BVP-4}
with initial data $\varphi_{\e, 0} \in H^3(\Omega)\cap H_0^1(\Omega)$ and
$\varphi_{\e, 1} \in H^2(\Omega)\cap H_0^1(\Omega)$.
Then,  for $0<\e<\min (T, r_0)$,
\begin{equation}\label{upper-40}
\aligned
& \int_0^T\!\!\! \int_{\partial\Omega} |\nabla u_\e|^2\, d\sigma dt\\
&\le C T \Big\{ \|\nabla \varphi_{\e, 0} \|^2_{L^2(\Omega)}
+ \| \varphi_{\e, 1} \|^2_{L^2(\Omega)} \Big\}\\
&\qquad
+ C T\e \Big\{ \|\mathcal{L}_\e (\varphi_{\e, 0} )\|^2_{L^2(\Omega)}
+ \|\nabla \varphi_{\e, 1} \|^2_{L^2(\Omega)} \Big\}\\
&\qquad
 + C T^{3} \e 
 \Big\{ \| \mathcal {L}_\e (\varphi_{\e, 0} )  \|^2_{H^1(\Omega)}
 + \| \mathcal{L}_\e (\varphi_{\e, 1} ) \|^2_{L^2 (\Omega)} \Big\}
 \\
 &\qquad
 + CT \e^3 \Big\{ \| \mathcal{L}_\e (\varphi_{\e, 0} ) \|^2_{H^1(\Omega)}
 + \|\mathcal{L}_\e ( \varphi_{\e, 1} ) \|^2_{L^2(\Omega)}  \Big\},
 \endaligned
\end{equation}
where $C$ depends only on $d$, $\mu$, $M$, and $\Omega$.
\end{thm}

\begin{proof}
Let $u_0$, $w_\e$ be the same as in Lemma \ref{lemma-4.2}.
Note that 
\begin{equation}\label{g-w}
\nabla w_\e =\nabla u_\e - (\nabla \Phi_\e) (\nabla u_0) - (\Phi_\e -x ) \nabla^2 u_0.
\end{equation}
It follows that
\begin{equation}\label{upper-41}
\aligned
& \int_0^T\!\!\! \int_{\partial\Omega} |\nabla u_\e|^2\, d\sigma dt\\
& \le C\int_0^T\!\!\! \int_{\partial\Omega} |\nabla w_\e|^2\, d\sigma dt
+ C \int_0^T\!\!\! \int_{\partial\Omega} |\nabla u_0 |^2\, d\sigma dt
+ C\e^2  \int_0^T\!\!\! \int_{\partial\Omega} |\nabla^2  u_0|^2\, d\sigma dt.
\endaligned
\end{equation}
To bound the first term in the right-hand side of (\ref{upper-41}),
we use (\ref{4.2-1}) as well as the fact   that $\varphi_1=\varphi_{\e, 1}$ and
$\mathcal{L}_0 (\varphi_0)=\mathcal{L}_\e (\varphi_{\e, 0 }) $ in $\Omega$.
The second term in the right-hand side of (\ref{upper-41}) is handled by Lemma \ref{lemma-4.1}.
Finally, to bound  the third  term, we use the inequality 
\begin{equation}\label{trace}
\int_{\partial\Omega}
|\nabla^2 u_0|^2\, d\sigma 
\le C \int_\Omega |\nabla^2 u_0|^2\, dx
+  C \int_\Omega |\nabla^2 u_0|\, |\nabla ^3 u_0|\, dx.
\end{equation}
To see (\ref{trace}), one chooses a vector field $h \in C_0^1 (\mathbb{R}^d; \mathbb{R}^d)$ such that
$\langle h, n \rangle\ge c_0>0 $ on $\partial\Omega$, and applies  the divergence theorem to  the integral 
$$
\int_{\partial\Omega} |\nabla^2 u_0|^2 \langle h, n \rangle\, d\sigma.
$$
\end{proof}

We also obtain a lower bound for $\|\nabla  u_\e\|_{L^2(S_T) }$.

\begin{thm}\label{theorem-4.2}
Assume that $A$ and $\Omega$ satisfies the same conditions as in Theorem \ref{theorem-4.1}.
Let $u_\e$ be a weak solution of \eqref{BVP-4}
with initial data $\varphi_{\e, 0} \in H^3(\Omega)\cap H_0^1(\Omega)$ and
$\varphi_{\e, 1} \in H^2(\Omega)\cap H_0^1(\Omega)$.
Then, if $T\ge C_0 r_0$ and $0<\e<r_0$,
\begin{equation}\label{lower-40}
\aligned
 & Tr_0^{-1}  \Big\{ \|\nabla \varphi_{\e, 0} \|^2_{L^2(\Omega)}
+ \| \varphi_{\e, 1} \|^2_{L^2(\Omega)} \Big\}\\
&
 \le  C \int_0^T\!\!\! \int_{\partial\Omega} |\nabla u_\e|^2\, d\sigma dt\\
 &\qquad
+ C T\e \Big\{ \|\mathcal{L}_\e (\varphi_{\e, 0} )\|^2_{L^2(\Omega)}
+ \|\nabla \varphi_{\e, 1} \|^2_{L^2(\Omega)} \Big\}\\
&\qquad
 + C T^{3} \e 
 \Big\{ \| \mathcal {L}_\e (\varphi_{\e, 0} )  \|^2_{H^1(\Omega)}
 + \| \mathcal{L}_\e (\varphi_{\e, 1} ) \|^2_{L^2 (\Omega)} \Big\}
 \\
 &\qquad
 + CT \e^3 \Big\{ \| \mathcal{L}_\e (\varphi_{\e, 0} ) \|^2_{H^1(\Omega)}
 + \|\mathcal{L}_\e ( \varphi_{\e, 1} ) \|^2_{L^2(\Omega)}  \Big\},
 \endaligned
\end{equation}
where $C$ depends only on $d$, $\mu$, $M$, and $\Omega$.
\end{thm}

\begin{proof}
The proof uses (\ref{g-w}) and  the fact that 
\begin{equation}\label{deter}
|\text{\rm det}
\left(\nabla \Phi_\e\right) | \ge c_0>0 \quad \text{  on }  \partial\Omega,
\end{equation}
which was proved in \cite{KLS-eigenvalue}.
Let $u_0$, $w_\e$ be the same as in Lemma \ref{lemma-4.2}.
It follows from (\ref{lower-b-0}), (\ref{g-w}) and (\ref{deter})  that 
\begin{equation}\label{lower-41}
\aligned
&  Tr_0^{-1}  \Big\{ \|\nabla \varphi_{\e, 0} \|^2_{L^2(\Omega)}
+ \| \varphi_{\e, 1} \|^2_{L^2(\Omega)} \Big\}\\
&\le C  \int_0^T\!\!\! \int_{\partial\Omega} |\nabla u_0|^2\, d\sigma dt\\
& \le C\int_0^T\!\!\! \int_{\partial\Omega} |\nabla u_\e|^2\, d\sigma dt
+ C \int_0^T\!\!\! \int_{\partial\Omega} |\nabla w_\e  |^2\, d\sigma dt
+ C\e^2  \int_0^T\!\!\! \int_{\partial\Omega} |\nabla^2  u_0|^2\, d\sigma dt.
\endaligned
\end{equation}
The last two terms in the right-hand side of (\ref{lower-41})
are treated exactly as in the proof of Theorem \ref{theorem-4.1}.
\end{proof}

\begin{proof}[\bf Proof of Theorem \ref{main-theorem-2}]
Let 
$$
\varphi_{\e, 0} =\sum_{\lambda_{\e, k} \le N} a_k \psi_{\e,  k} 
\quad 
\text{ and }
\quad
\varphi_{\e, 1} =\sum_{\lambda_{\e, k} \le N} b_k \psi_{\e, k},
$$
where $\{ \psi_{\e, k} \}$ forms an orthonormal basis for $L^2(\Omega)$, $\psi_{\e, k} \in H_0^1(\Omega)$ and
$\mathcal{L}_\e (\psi_{\e, k} ) =\lambda_{\e, k} \psi_{\e, k}$ in $\Omega$.
Then
\begin{equation} \label{last-1}
\|\nabla \varphi_{\e, 0} \|_{L^2(\Omega)}^2
+\|\varphi_{\e, 1} \|_{L^2(\Omega)}^2
\sim \sum_{\lambda_{\e, k} \le N}
\Big\{ |a_k|^2 \lambda_{\e, k} + |b_k|^2 \Big\}.
\end{equation}
Also, note that
\begin{equation}\label{last-2}
\aligned
\| \mathcal{L}_\e ( \varphi_{\e, 0})  \|_{L^2(\Omega)}^2
+\| \nabla \varphi_{\e, 1} \|_{L^2(\Omega)}^2
 & \le C  \sum_{\lambda_{\e, k} \le N}
\Big\{ |a_k|^2 \lambda_{\e, k}^2  + |b_k|^2 \lambda_{\e, k}  \Big\}\\
&  \le C N   \sum_{\lambda_{\e, k} \le N}
\Big\{ |a_k|^2 \lambda_{\e, k}  + |b_k|^2  \Big\},
\endaligned
\end{equation}
and that
\begin{equation}\label{last-3}
\aligned
\| \mathcal{L}_\e ( \varphi_{\e, 0})  \|_{H^1(\Omega)}^2
+\| \mathcal{L}_\e (\varphi_{\e, 1})  \|_{L^2(\Omega)}^2
 & \le C  \sum_{\lambda_{\e, k} \le N}
\Big\{ |a_k|^2 \lambda_{\e, k}^3  + |b_k|^2 \lambda_{\e, k}^2  \Big\}\\
&  \le C N^2   \sum_{\lambda_{\e, k} \le N}
\Big\{ |a_k|^2 \lambda_{\e, k}  + |b_k|^2  \Big\}.
\endaligned
\end{equation}
In view of Theorem \ref{theorem-4.1} we obtain 
$$
\aligned
\int_0^T\!\!\!\int_{\partial\Omega}
|\nabla u_\e|^2\, d\sigma dt
&\le C T \Big\{ 1+ \e N + T^2 \e N^2
+ \e^3 N^2 \Big\}
\Big\{ \|\nabla \varphi_{\e, 0} \|^2_{L^2(\Omega)}
+ \|\varphi_{\e, 1} \|_{L^2(\Omega)}^2 \Big\}\\
&\le CT  \Big\{ \|\nabla \varphi_{\e, 0} \|^2_{L^2(\Omega)}
+ \|\varphi_{\e, 1} \|_{L^2(\Omega)}^2 \Big\},
\endaligned
$$
if $N\le C_0 T^{-1} \e^{-1/2}$.
This gives the inequality (\ref{upper-b}).
The inequality (\ref{lower-b}) follows from Theorem \ref{theorem-4.2} in a similar manner.
Indeed, by Theorem \ref{theorem-4.2} and \eqref{last-1}-\eqref{last-3}, if $T\ge T_0$,
$$
\aligned
& T  \Big\{ \|\nabla \varphi_{\e, 0} \|^2_{L^2(\Omega)}
+ \| \varphi_{\e, 1} \|^2_{L^2(\Omega)} \Big\}\\
&
 \le  C \int_0^T\!\!\! \int_{\partial\Omega} |\nabla u_\e|^2\, d\sigma dt
 +CT \Big \{ \e N +T^2\e N^2
+ \e^3 N^2  \Big\}\Big\{ \|\nabla \varphi_{\e, 0} \|^2_{L^2(\Omega)}
+ \| \varphi_{\e, 1} \|^2_{L^2(\Omega)} \Big\},
\endaligned
$$
where $T_0$ and $C$ depends only on $d$, $\mu$, $M$, and $\Omega$.
As a result, we obtain (\ref{lower-b} if $N\le c_0 T^{-1} \e^{-1/2}$, where $c_0=c_0(d, \mu, M, \Omega)>0 $ is so small that
$$
C \big\{ \e N + T^2\e N^2 +\e^3 N^2 \big\} \le (1/2).
$$
This completes the proof.
\end{proof}

\begin{remark}
{\rm 
Let $\Gamma$ be a subset of $\partial\Omega $.
Suppose that there exist $T>0$ and $c_0>0$ such that  the inequality 
\begin{equation}\label{p-lower-0}
c_0 \Big\{ \|\nabla \varphi_0 \|^2_{L^2(\Omega)}
+ \|\varphi_1 \|^2_{L^2(\Omega)} \Big\}
\le \frac{1}{T} \int_0^T\!\!\! \int_\Gamma |\nabla u_0|^2\, d\sigma dt
\end{equation}
holds for solutions  $u_0$ of the homogenized problem (\ref{BVP-3-0}).
It follows from the proof of Theorem \ref{main-theorem-2} that if $N \le \delta \e^{-1/2}$ and
$\delta=\delta(c_0, T, \Omega, A)>0$ is sufficiently small, 
the inequality 
\begin{equation}\label{p-lower}
c \Big\{ \|\nabla \varphi_{\e, 0} \|^2_{L^2(\Omega)}
+ \|\varphi_{\e, 1}  \|^2_{L^2(\Omega)} \Big\}
\le \frac{1}{T} \int_0^T\!\!\!\int_\Gamma 
 |\nabla u_\e|^2\, d\sigma dt
\end{equation}
}
holds for solutions $u_\e$ 
 of (\ref{BVP-0}) with initial data $(\varphi_{\e, 0}, \varphi_{\e, 1}) $  in $\mathcal{A}_N \times \mathcal{A}_N$.
\end{remark}

Given $(\theta_{\e, 0}, \theta_{\e, 1})\in L^2(\Omega) \times H^{-1}(\Omega)$,
to find a control $g_\e\in L^2(S_T)$ such that the solution of (\ref{E-0}) satisfies 
the projection  condition (\ref{P-C}),
one considers the functional, 
$$
J_\e (\varphi_{\e, 0}, \varphi_{\e, 1})
=-\langle  \theta_{\e, 1},  u_\e (x, 0)\rangle_{H^{-1}(\Omega) \times H_0^1(\Omega)} 
+ \int_\Omega  \theta_{\e, 0} \partial_t u_\e (x, 0)  dx
+\frac12 \int_0^T\!\!\!\int_{\partial\Omega}
\left(\frac{\partial u_\e}{\partial\nu_\e}\right)^2 d\sigma dt,
$$
where $\frac{\partial u_\e}{\partial \nu_\e} =n_i a_{ij}(x/\e) \frac{\partial u_\e}{\partial x_j}$
denotes the conormal derivative associated with $\mathcal{L}_\e$, and
$u_\e$ is the solution of
\begin{equation}\label{dual-p}
\left\{
\aligned
& (\partial^2_t +\mathcal{L}_\e) u_\e =0  \quad  \text{ in } \Omega_T,\\
& u_\e =0 \quad \text{ on } S_T\\
& u_\e (x, T) =\varphi_{\e, 0}, \quad \partial_t u_\e (x, T)= \varphi_{\e, 1} \quad \text{ for } x\in \Omega.
\endaligned
\right.
\end{equation}
Since the time is reversible in the wave equation, 
it follows from Theorem \ref{main-theorem-2} that
if $( \varphi_{\e, 0}, \varphi_{\e, 1}) \in \mathcal{A}_N \times \mathcal{A}_N$ and  $N\le \delta \e^{-2/3}$, 
\begin{equation}\label{control-b}
\aligned
c\Big\{ \| \nabla \varphi_{\e, 0} \|^2_{L^2(\Omega)}
+ \|\varphi_{\e, 1} \|^2_{L^2(\Omega)} \Big\}
 & \le \int_0^T\!\!\! \int_{\partial\Omega}
\left(\frac{\partial u_\e}{\partial \nu_\e} \right)^2  d\sigma dt\\
& \le 
C \Big\{ \| \nabla \varphi_{\e, 0} \|^2_{L^2(\Omega)}
+ \|\varphi_{\e, 1} \|^2_{L^2(\Omega)}\Big\},
\endaligned
\end{equation}
where $C>0$ and $c>0$ are independent of $\e>0$.
As a functional on $\mathcal{A}_N \times \mathcal{A}_N \subset  H^1_0(\Omega) \times L^2(\Omega)$,
$J_\e$ is continuous, strictly convex, and satisfies the coercivity estimate, 
$$
J_\e (\varphi_{\e, 0}, \varphi_{\e, 1})
\ge 
c\Big\{ \| \nabla \varphi_{\e, 0} \|^2_{L^2(\Omega)}
+ \|\varphi_{\e, 1} \|^2_{L^2(\Omega)} \Big\}
-C \Big\{ \|\theta_{\e, 0} \|_{L^2(\Omega)}^2
+ \|\theta_{\e, 1}\|_{H^{-1}(\Omega)}^2 \Big\}.
$$
This implies that $J_\e$ possesses a unique minimum  $J_\e (\phi_0, \phi_1)$ on $\mathcal{A}_N \times \mathcal{A}_N$.
Let $w_\e$ be the solution of (\ref{dual-p}) with data $( w_\e (x, T), \partial_t w_\e (x, T)) =(\phi_{ 0}, \phi_{1})$.
By the first variational principle,
\begin{equation}\label{c-1}
-\langle  \theta_{\e, 1},  u_\e (x, 0)\rangle_{H^{-1}(\Omega) \times H_0^1(\Omega)} 
+ \int_\Omega  \theta_{\e, 0} \partial_t u_\e (x, 0)  dx
+ \int_0^T\!\!\!\int_{\partial\Omega}
\frac{\partial w_\e}{\partial \nu_\e} \cdot
\frac{\partial u_\e}{\partial \nu_\e} \, d\sigma dt
=0,
\end{equation}
for any solution  $u_\e$ of (\ref{dual-p}) with data $(\varphi_{\e, 0}, \varphi_{\e, 1}) \in \mathcal{A}_N\times  \mathcal{A}_N$.
As a result, the function $g_\e=\frac{\partial w_\e}{\partial \nu_\e}$ is a control that gives (\ref{P-C}).
Indeed, let $v_\e$ be the solution of (\ref{E-0}) with $g_\e=\frac{\partial w_\e}{\partial\nu_\e}$,
then
$$
\aligned
& \langle \partial_t v_\e (\cdot, T), \varphi_{\e, 0} \rangle_{H^{-1}(\Omega) \times H_0^1(\Omega)}
-\int_\Omega v_\e (x, T) \varphi_{\e, 1} (x)\, dx\\
&=
\langle \theta_{\e, 1}, u_\e (\cdot, T) \rangle_{H^{-1}(\Omega) \times H^1_0(\Omega)}
-\int_\Omega \theta_{\e, 0} \partial_t u_\e (x, 0)\, dx
-\int_0^T \!\!\! 
\int_{\partial\Omega}
g_\e \frac{\partial u_\e}{\partial \nu_\e} \, d\sigma dt\\
&=0
\endaligned
$$
for any $(\varphi_{\e, 0}, \varphi_{\e, 1} ) \in \mathcal{A}_N \times \mathcal{A}_N$.
This shows that $P_N v_\e (x, T)=0$ and
$P_N \partial_t v_\e (x, T)=0$ for $x\in \Omega$.
One may also use (\ref{c-1}) to show that 
among all controls that give (\ref{P-C}),
$g_\e=\frac{\partial w_\e}{\partial\nu_\e}$ has the minimal $L^2(S_T)$ norm.

Finally, using $J_\e (\phi_0 , \phi_1 ) \le J(0, 0)=0$ and (\ref{control-b}),
one may deduce that
$$
\int_0^T\!\!\! \int_{\partial\Omega}
|g_\e|^2 d\sigma dt
\le C \Big\{ \| P_N \theta_{\e, 0} \|_{L^2(\Omega)}^2
+ \| P_N \theta_{\e, 1}\|_{H^{-1}(\Omega)}^2 \Big\}.
$$
By a duality argument \cite{Castro-1999} and (\ref{control-b}),
one may also show that
$$
c \Big\{ \| P_N \theta_{\e, 0} \|_{L^2(\Omega)}^2
+ \| P_N \theta_{\e, 1}\|_{H^{-1}(\Omega)}^2 \Big\}
\le 
\int_0^T\!\!\! \int_{\partial\Omega}
|g_\e|^2   d\sigma dt.
$$
We omit the details and refer the reader to \cite{Castro-1999} for the one-dimensional case.

 \bibliographystyle{amsplain}
 
\bibliography{Lin-Shen-2019.bbl}

\bigskip

\begin{flushleft}

Fanghua Lin, 
Courant Institute of Mathematical Sciences, 
251 Mercer Street, 
New York, NY 10012, USA.

Email: linf@cims.nyu.edu

\bigskip

Zhongwei Shen (corresponding author),
Department of Mathematics,
University of Kentucky,
Lexington, Kentucky 40506,
USA.

E-mail: zshen2@uky.edu
\end{flushleft}

\bigskip

\medskip

\end{document}